\documentclass[12pt]{article}

\usepackage{authblk}

\usepackage{algpseudocode}
\usepackage{algorithm}
\usepackage{subcaption}

\usepackage{amsfonts}
\usepackage{graphicx}
\usepackage{tabularx}
\usepackage{array}
\usepackage[usenames,dvipsnames]{color}
\usepackage{amsmath}
\usepackage{amsthm}
\usepackage{amssymb}
\usepackage{fullpage}
\usepackage[dvipsnames]{xcolor}
\usepackage{tikz}
\usepackage{listings}
\usepackage{dsfont}
\usepackage{enumerate}
\usepackage{romannum}

\newtheorem{theorem}{Theorem}[section]
\newtheorem{proposition}[theorem]{Proposition}

\newtheorem{conjecture}[theorem]{Conjecture}
\newtheorem{problem}[theorem]{Problem}

\newtheorem{lemma}[theorem]{Lemma}
\newtheorem{corollary}[theorem]{Corollary}
\theoremstyle{definition}

\def\epsilon{\varepsilon}
\DeclareMathOperator{\dist}{dist}

\title{Greedily Constructing Small Quasi-Kernels}
\author[1]{Alexander Clow\,\thanks{Supported by the Natural Sciences and Engineering Research Council of Canada (NSERC) through PGS D-601066-2025}}

\affil[1]{ \small{Department of Mathematics, Simon Fraser University}}
\date{}

\begin{document}
\pagenumbering{arabic}
\maketitle

\begin{abstract}

    In a digraph $D$,
    a quasi-kernel is an independent set $Q$ such that for every vertex $u$, there is a vertex $v \in Q$ satisfying $\dist(v,u)\leq 2$.
    In 1974 Chv\'atal and Lov\'asz showed every digraph contains a quasi-kernel.
    In 1976, P. L. Erd\H{o}s and Sz\'ekely conjectured that every sourceless digraph has a quasi-kernel of order at most $\frac{n}{2}$.
    Despite significant recent attention by the community the problem remains far from solved, with no bound of the form $(1-\epsilon)n$ known.
    We introduce a polynomial time algorithm which greedily constructs a small quasi-kernel.
    Using this algorithm we show that if $D$ is a $\Vec{K}_{1,d}$-free digraph, then $D$ has a 
    quasi-kernel of order at most $\frac{(d^2 - 2d + 2)n}{d^2-d+1}$.
    By refining this argument we prove that for any $D$ with maximum out-degree $3$ this algorithm constructs a 
    quasi-kernel of order at most ${4n}/{7}$.
    Finally,
    we consider the problem in digraphs forbidding certain orientation of short cycles
    as subgraphs,
    concluding that all orientations $D$ of a graph $G$ with girth at least $7$
    have a quasi-kernel of order at most $\frac{(d^2+4)n}{(d+2)^2}$, where
    $d$ is the maximum out-degree of $D$.
\end{abstract}

\section{Introduction}

\subsection{Background}

We consider directed graphs (digraphs), which may have multiedges. 
Loops and multiple edges of the form $(u,v)$ are not considered, as they have no effect on the questions we study.
Notice that this allows for edges $(u,v)$ and $(v,u)$ to both be present.
An oriented graph is a digraph such that for all pairs $u,v$
at most one of the edges $(u,v)$ or $(v,u)$ is present.
Given $u$ and $v$ in a digraph $D$ we let $\dist(v,u)$ denote the length of a shortest directed path from $v$ to $u$.
If no such path exists, then we say $\dist(v,u) = \infty$.
A source is a vertex with in-degree zero and a sink is a vertex with out-degree zero.
All other digraph notation we use is standard or defined explicitly.

Let $D$ be a digraph,
a \emph{kernel} is an independent set $K$ such that every vertex is in $K$ or has an in-neighbour in $K$.
That is, $K$ is both an independent set and a directed dominating set.
Not every digraph contains a kernel, see any odd directed cycle.
In some sense this is a minimal obstruction for the existence of a kernel, 
since Richardson \cite{richardson1953solutions} 
proved that every digraph without a directed odd cycle contains a kernel.
This is a departure from undirected graphs, where every maximal independent set is necessarily a dominating set.
Deciding if a digraph has a kernel was shown to be NP-hard by Chv{\'a}tal \cite{chvatal1973computational}.

An object related to kernels is quasi-kernels.
A \emph{quasi-kernel} is an independent set $Q$ such that for every vertex $u$ in $D$ there exists a $v\in Q$ such that $\dist(v,u) \leq 2$.
That is $Q$ is both an independent set and a directed distance-$2$-dominating set.
Unlike with kernels, Chv\'atal and Lov\'asz \cite{chvatal1974every} showed that every digraph contains a quasi-kernel.
The proof here is algorithmic:
Let $D$ be a smallest digraph without a quasi-kernel. 
Let $v$ be a vertex in $D$, and let $H = D - N^{+}[v]$.
By induction there exists a quasi-kernel $K$ in $H$.
If $K\cap N^-(v)\neq \emptyset$, then $K$ is a quasi-kernel in $D$.
Otherwise $K\cup \{v\}$ is a quasi-kernel in $D$.

Hence, by repeatedly applying this procedure one can construct a quasi-kernel efficiently.
Furthermore, this proves that for every vertex $v$, there is a quasi-kernel $Q$ such that $Q \cap N^{+}(v)= \emptyset$.
Interestingly, Croitoru \cite{croitoru2015note} proved that deciding if there exists a quasi-kernel $Q$ such that $v \in Q$
is NP-complete.
In a similar vein, 
it was shown by Langlois, Meunier, Rizzi, and Vialette \cite{langlois2022algorithmic}
that deciding if a digraph has two disjoint quasi-kernels is NP-complete, even for digraphs with maximum out-degree at most $6$.
Moreover, it was shown in \cite{langlois2022algorithmic} 
that the decision problem related to finding a smallest quasi-kernel
is NP-complete
even for acyclic orientations of cubic
graphs.

An obvious obstruction for a digraph $D$ to have a small quasi-kernel is if $D$ contains a large number of sources.
This observation led to Erd\H{o}s and Sz\'ekely conjecturing in 1976 that every sourceless digraphs admits a small quasi-kernel.
The conjecture was initially not published and was communicated by
word of mouth until it finally appeared in 2010 \cite{erdos2010two}.

\begin{conjecture}[The Small Quasi-Kernel Conjecture]\label{Conj: small quasi}
 If $D$ is a sourceless $n$ vertex digraph, then
$D$ contains a quasi-kernel with order at most $\frac{n}{2}$.
\end{conjecture}

If correct then this conjecture is best possible for infinitely many digraphs.
To see this observe that the bound is tight when $D$ is a disjoint union of directed $4$-cycles.
Moreover, the bound is tight asymptotically for arbitrarily large strongly connected digraphs, 
see Example~17 in \cite{erdHos2025small}.

Conjecture~\ref{Conj: small quasi} has been proven for some classes of digraphs.
It is a corollary of Chv\'atal and Lov\'asz's work that any digraph with independence number at most $\frac{n}{2}$ contains a small quasi-kernel.
So
tournaments contain small quasi-kernels.
Note this was independently proven much earlier by 
Landau \cite{landau1953dominance}.
Nguyen, Scott, and Seymour \cite{nguyen2024distant}
proved that Conjecture~\ref{Conj: small quasi} holds for split digraphs.
This extends two earlier partial results for split digraphs, see \cite{ai2023results,langlois2025quasi}.
Kostochka, Luo, and Shan \cite{kostochka2022towards}
proved that Conjecture~\ref{Conj: small quasi} holds
for orientations of $4$-colourable graphs.
Also, Ai, Gerke,Gutin, Yeo, and Zhou \cite{ai2023results} proved that Conjecture~\ref{Conj: small quasi} holds 
when forbidding an orientation of the claw as an induced subgraph.
Earlier work by Heard and Huang \cite{heard2008disjoint} proved Conjecture~\ref{Conj: small quasi}
for semicomplete multipartite, quasi-transitive, and locally semicomplete
digraphs.

The best upper bound on the size of a smallest quasi-kernel in a sourceless digraph
comes from Spiro \cite{spiro2024generalized}, who showed every sourceless digraph has a quasi-kernel
of order at most $n - \frac{1}{4}\sqrt{n\log(n)}$.
The proof of this bound is existential, and it does not provide an algorithm which is guaranteed
to output a quasi-kernel of size at most $n - \frac{1}{4}\sqrt{n\log(n)}$.

\subsection{Our results}

Our primary contribution
is to define a new polynomial time algorithm, see Algorithm~\ref{alg:main}, for constructing a quasi-kernel in a sourceless digraph.
In spirit this is a greedy 
algorithm operating on the same principles as
Chv\'atal and Lov\'asz's proof for the existence of quasi-kernels.
However, this novel approach provides some distinct advantages 
for constructing, as well as 
for proving the existence of, small quasi-kernels
compared to the naive existence algorithm.
In particular,
if the structure of the input digraph is well understood, 
then we can guarantee the quasi-kernel outputted by this algorithm has bounded size.

The rest of the paper is spent 
analyzing how this algorithm interacts with various digraph structures.
To more easily describe the relevant structure we introduce the following notation.
Let $D$ be a digraph and $X$ a vertex subset in $D$,
then $S(X)$ denotes the set of sources in $D - N^{+}[X]$.
When $X = \{v\}$, we write $S(v)$ rather than $S(\{v\})$.
A set of digraphs $\mathcal{H}$ is a hereditary class if $\mathcal{H}$ is closed on taking induced subgraphs.

\begin{theorem}\label{Thm: Main CounterStructure}
    Let $t\geq 1$ be a constant, and let 
    $\mathcal{H}$ be a hereditary class of digraphs such that for all sourceless digraphs $D \in \mathcal{H}$
    there exists a vertex $v \in V(D)$ where $v \notin N^+(S(v))$ and
    \[
    t|N^+(v) \cup N^+(S(v))| \geq |S(v)|+1,
    \]
    or  $v \in N^+(S(v))$ and
    \[
    t|N^+(v) \cup N^+(S(v))| \geq |S(v)|.
    \]
    Then, there is a $O(n^4)$ time algorithm
    which takes any sourceless digraph $D \in \mathcal{H}$ of order $n$
    and returns a quasi-kernel $Q$ in $D$ such that
    $|Q| \leq \frac{tn}{t+1}$.
\end{theorem}

We say a digraph $D$ is $H$-free if $H$ is not an induced subgraph of $D$.
Let $\Vec{K}_{1,d}$ be the orientation of $K_{1,d}$ with all edges pointing away from the high degree vertex,
see Figure~\ref{fig:small digraphs}.
Notice that Theorem~\ref{Thm: Main CounterStructure}
implies that
if $D$ is a $\Vec{K}_{1,d}$-free digraph,
then $D$ contains a quasi-kernel of bounded size.

\begin{corollary}\label{Coro: Induced Stars}
    Let $d\geq 3$ be an integer.
    If $D$ is a sourceless $\Vec{K}_{1,d}$-free digraph,
    then $D$ contains a quasi-kernel of order at most $\frac{(d^2 - 2d + 2)n}{d^2-d+1}$.
\end{corollary}

\begin{proof}
    For any digraph $D$, and vertex $v \in V(D)$, 
    $S(v)$ being the set of sources in $D - N^+[v]$, implies $S(v)$ is an independent set in $D$.
    Thus each vertex in $N^{+}(v)$ can contribute at most $d-1$ 
    vertices to $S(v)$ if $D$ is $\Vec{K}_{1,d}$-free.
    Let $d$ be the least integer such that $D$ is $\Vec{K}_{1,d}$-free,
    then $\Delta^+(D)\geq d-1$, otherwise, $D$ is trivially $\Vec{K}_{1,d-1}$-free.
    Select $v$ to be a vertex with maximum out-degree, then 
    $\frac{(d-1)\deg^+(v) + 1}{\deg^+(v)}|N^+(v) \cup N^+(S(v))| \geq  |S(v)|+1$.
    Treating $d$ as fixed, $\frac{(d-1)\deg^+(v) + 1}{\deg^+(v)}$ decreases to $d-1$ monotonically as $\deg^+(v)$ tends to infinity.
    Since $\deg^+(v)\geq d-1$,
    we conclude that 
    $t|N^+(v) \cup N^+(S(v))| \geq  |S(v)|+1$
    where $t = \frac{d^2-2d+2}{d-1}$.
    The result follows by Theorem~\ref{Thm: Main CounterStructure}
\end{proof}

\begin{figure}[h!]
    \centering
    \includegraphics[scale = 1.0]{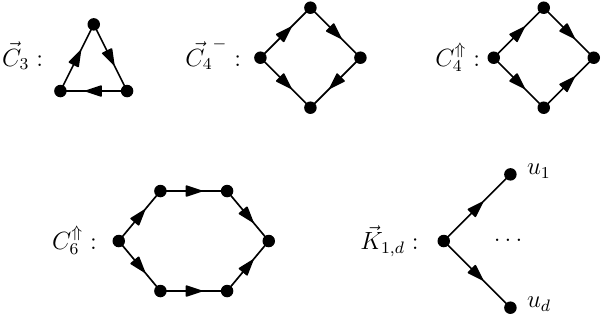}
    \caption{Some orientations of $C_3,C_4,C_6$, and $K_{1,t}$ that are of interest.}
    \label{fig:small digraphs}
\end{figure}

Note that it was shown in \cite{ai2023results}
that every sourceless $\Vec{K}_{1,3}$-free digraph has a quasi-kernel of order at most $\frac{n}{2}$.
An immediate consequence of this is that all sourceless digraphs with maximum out-degree $2$ have a small quasi-kernel.
Observe that in the same way
Corollary~\ref{Coro: Induced Stars} provides a new bound on the size of a smallest quasi-kernel in 
digraphs with bounded maximum out-degree.
To demonstrate how our work can be used to improve these bounds further for digraphs with bounded maximum out-degree, 
we provide an improved bound for digraphs with maximum out-degree $3$.

\begin{theorem}\label{Thm: Out-degree 3}
    If $D$ is a sourceless digraph with maximum out-degree $\Delta^+(D)\leq 3$,
    then $D$ contains a quasi-kernel of order at most $\frac{4n}{7}$.
\end{theorem}

Next, we show a similar result for digraphs with restricted orientations of $3,4$, and $6$ cycles.
See Figure~\ref{fig:small digraphs} for drawings of these orientations.

\begin{theorem}\label{Thm: C_3,C_4,C_6}
    Let $d\geq 2$ be a fixed integer.
    If $D$ is a sourceless digraph with maximum out-degree $\Delta^+(D) \leq d$, and no $\Vec{C}_3, \Vec{C}^{-}_4,C_4^{\Uparrow},$ or $C_6^{\Uparrow}$ as a subgraph,
    then $D$ has a quasi-kernel of order at most $\frac{(d^2+4)n}{(d+2)^2}$.
    Moreover, such a quasi-kernel can be found in $O(n^2)$ time.
\end{theorem}

It is natural to compare the results from 
Corollary~\ref{Coro: Induced Stars} and Theorem~\ref{Thm: C_3,C_4,C_6}.
To this end we provide Table~\ref{table:ratios}, so that the reader can get a sense of how these bounds compare to one another as $d$ grows.

\renewcommand{\arraystretch}{1.5}
\begin{table}[!h]
\centering
\begin{tabular}{| c | c | c |} 
\hline \hline
$d$ & $K_{1,d}$-free & $\Delta^+(D) \leq d$, and no $\Vec{C}_3, \Vec{C}^{-}_4,C_4^{\Uparrow},$ or $C_6^{\Uparrow}$ subgraph \\ \hline \hline 
$4$ &  $\frac{10}{13}$ & $\frac{5}{9}$\\ \hline
$5$ & $\frac{17}{21}$ &  $\frac{29}{49}$\\ \hline
$6$ & $\frac{26}{31}$ &  $\frac{5}{8}$\\ \hline
$7$ & $\frac{37}{43}$  &  $\frac{53}{81}$ \\ \hline
$8$ & $\frac{50}{57}$  & $\frac{17}{25}$\\ \hline
$25$ & $\frac{577}{601}$ & $\frac{629}{729}$\\ \hline 
$50$ & $\frac{2402}{2451}$ &  $\frac{313}{338}$ \\ \hline
$100$ & $\frac{9802}{9901}$ & $\frac{2501}{2601}$ \\ \hline \hline
\end{tabular}
\caption{
For each value of $d$ we provide the constant factor on the upper bound of the order of a smallest quasi-kernel
from Corollary~\ref{Coro: Induced Stars} (left) and Theorem~\ref{Thm: C_3,C_4,C_6} (right).
}
\label{table:ratios}
\end{table}

\section{Constructing Small Quasi-Kernels}

This section outlines the algorithm we provide for building small quasi-kernels,
see Algorithm~\ref{alg:main}.
We provide a proof that the algorithm outputs a quasi-kernel, see Lemma~\ref{Lemma: Alg works},
then analyze
the size of the resulting quasi-kernels.
See for instance the proof of Theorem~\ref{Thm: Main CounterStructure}.
Finally, we show that this analysis is insufficient to prove Conjecture~\ref{Conj: small quasi}
for all digraphs, by demonstrating digraphs
where the bounds in 
Theorem~\ref{Thm: Main CounterStructure}
fail to accurately predict the correct order of a smallest quasi-kernel,
see Proposition~\ref{Prop: Alg not perfect}.

\begin{algorithm}[H]
\caption{An algorithm for constructing a quasi-kernel.}
\label{alg:main}
\begin{algorithmic}
\Require $D$ a sourceless digraph, $v$ a vertex in $D$.
\State $H = D - (N^+{[v]} \cup N^+[S(v)])$
\State $K = S(v)$.

\While{$H$ contains a source}
    \State Let $I$ be the set of isolated vertices in $H$.
    \State $H \gets (H - I)$

     \If{$H$ has a source}
        \State Let $x$ be a source in $H$ with $\deg^+_H(x)\geq 1$.
        \State $K \gets (K \cup \{x\})$
        \State $H \gets (H -N^+[x])$
    \EndIf

\EndWhile

\State Let $R$ be a smallest quasi-kernel in $H$.

\If{ $v \notin N^+(K \cup R)$}
    \State $K \gets (K\cup\{v\})$
\EndIf

\State $Q = K \cup R$.

\State \Return $Q$

\end{algorithmic}
\end{algorithm}

To start we will prove that the output of Algorithm~\ref{alg:main} is in fact a quasi-kernel in the input digraph.

\begin{lemma}\label{Lemma: Alg works}
    Let $D$ be a sourceless digraph and $v$ a vertex in $D$.
    If $D$ and $v$ are the inputs in Algorithm~\ref{alg:main}, 
    then the output of Algorithm~\ref{alg:main} is a quasi-kernel $Q$ in $D$.
\end{lemma}

\begin{proof}
    Let $D$ be a sourceless digraph and $v$ a vertex in $D$.
    Let $Q$ be the output of Algorithm~\ref{alg:main} with inputs $D$ and $v$.
    
    We begin by showing $Q$ is an independent set.
    Next, we show that for all $w \in V(D)$ there exists a $u\in Q$ such that $\dist(u,w)\leq 2$.
    Suppose the Algorithm~\ref{alg:main} ended after $i$ iterations of the while loop.
    Notice that $i\geq 0$.

\vspace{0.25cm}
\noindent\underline{Claim.1:} $Q$ is an independent set.
\vspace{0.25cm}  

    If $i = 0$, then $V(D) = N^+{[v]} \cup N^+[S(v)]$.
    In this case, $K = S(v)$ if $v \in N^+(v)$, and $K = S(v) \cup \{v\}$ if $v \notin N^+(v)$.
    Since $S(v)$ is the set of sources in $D-N^+[v]$, it is clear that $S(v)$ is an independent set.
    Moreover, $S(v) \cup \{v\}$ is an independent set unless $v \in N^+(S(v))$.
    Thus, $K$ is an independent set when $i = 0$.
    Observing that when $i = 0$, $V(H) = \emptyset$ we note that $R = \emptyset$
    implying that $Q = K$ is an independent set.

    Now consider $i\geq 1$.
    Since the set $Q$ is the union of $K$, 
    which is monotone increasing for each round through the while loop,
    and the set $R$,
    if $Q$ is not independent then either 
    $R$ is not an independent set, 
    $K$ is not an independent set,
    or there is an edge between a vertex of $K$ and a vertex of $R$.
    Observe that $R$ is by chosen to be a quasi-kernel of an induced subgraph of $D$.
    Hence, $R$ is necessarily an independent set.
    We consider the two remaining cases separately.

    If $K$ is not an independent set then there is a smallest $j$ such that $K$ is independent after the 
    $j^{\text{th}}$ iteration of the while loop $K$ is not an independent set, or $K$ is an independent set following the while loop,
    but not an independent set at the end of the algorithm.
    This second option is only possible if $v \in K$, $K - \{v\}$ is an independent set, 
    but $K$ is not an independent set.
    Notice that all vertices in $K$ belong to $V(D)\setminus N^+[v]$,
    while $v \in K$ only if $v \notin N^+(K)$.
    Thus, if the second case is not possible.
    Suppose then, for contradiction, that such an integer $j$ exists.
    By the same argument as the $i=0$ case, we note that $K$ is an independent set before the first iteration of the while loop, 
    so we can suppose without loss of generality that $K$ begins the $j^{\text{th}}$ iteration of the while loop as an independent set.

    The  $j^{\text{th}}$ iteration of the while loop begins by deleting isolated vertices.
    If the removal of isolated vertices leaves $H$ sourceless, then 
    either no vertex is added to $K$, or $v$ is the only vertex added to $K$.
    If no vertex is added, since $K$ began the  $j^{\text{th}}$ iteration of the loop as 
    and independent set, $K$ ends the $j^{\text{th}}$ iteration of the loop as 
    and independent set.
    This contradicts our choice of $j$, so we suppose that after deleting isolated vertices in $H$, there exists a source in $H$.

    Let $x$ be a source in $H$, after the deletion of isolated vertices during the $j^{\text{th}}$ iteration of the loop.
    Since $x$ is a source that was not deleted as an independent vertex, $\deg^+_H(x) \geq 1$.
    Suppose without loss of generality that $x$ is the vertex added to $K$ in the $j^{\text{th}}$ iteration of the while loop.
    Since the out-neighbours of each vertex already in $K$ are deleted from $H$,
    $N^-(x) \cap K = \emptyset$.
    Moreover, $x \notin N^+[v]$.
    Since every vertex in $K$, distinct from $x$, was a source in $D-N^+[v]$
    or in another prior iteration of $H$ which included $x$,
    we note that $N^+(x) \cap K = \emptyset$.
    Hence, $K$ is an independent set at the end of the  $j^{\text{th}}$ iteration of the loop.
    Again, this contradicts our choice of $j$, implying that the final choice of $K$ is in fact an independent set.

    The final possibility is that there exists an edge between a vertex of $K$ and a vertex of $R$.
    As before, if $v \in K$, then $v$ has no neighbours in $K\cup R$.
    Suppose then there is a vertex $x \in K\setminus \{v\}$ and $y \in R$
    such that $x$ and $y$ are adjacent.
    Since $R$ is a set of vertices in $H$, $y \notin N^+[x]$ given these vertices were removed from $H$ when $x$ was added to $K$.
    Similarly, since $R$ is a set of vertices in $H$, $x \notin N^+[y]$, since $x$ was sourceless when it was added to $K$.

    We conclude that $Q$ is an independent set as required.
\hfill $\diamond$
\vspace{0.25cm}

\vspace{0.25cm}
\noindent\underline{Claim.2:}  For all $w \in V(D)$ there exists a vertex $u\in Q$ such that $\dist(u,w)\leq 2$.
\vspace{0.25cm}  

    For contradiction suppose there exists a vertex $w \in V(D)$ such that for all $u \in Q$, $\dist(u,w) > 2$.
    To begin consider the case where $w$ is a vertex in the final iteration of $H$.
    In this case $R \subseteq Q$ is a quasi-kernel in $H$ implying that there exists a vertex $u \in R \subseteq Q$ 
    such that $\dist_D(u,w)\leq \dist_H(u,w) \leq 2$.
    Since this is a contradiction we suppose that $w$ is not a vertex in the final iteration of $H$.

    Then $w$ is a vertex in $D$ not in $H$.
    This implies $w \in N^+[K]$, or $w$ was removed as an isolated vertex, or $w \in N^+[v]$ and $v \notin K$.
    If $w \in N^+[K]$, then there exist a vertex $u \in K \subseteq Q$ 
    such that $\dist_D(u,w) = 1$.
    Since $D$ is sourceless, if $w$ was removed as an isolated vertex, then this implies the neighbours of $w$ are out neighbours of $K$, 
    so exist a vertex $u \in K \subseteq Q$ 
    such that $\dist_D(u,w) = 2$.
    Otherwise, $w \in N^+(v)$ and $v \notin K$.
    Recall that $v \notin K$ implies that $N^-(v) \cap Q \neq \emptyset$,
    so exist a vertex $u \in Q$ 
    such that $\dist_D(u,w) \leq 2$. 
    Since this is a contradiction for all $w \in V(D)$ there exists a vertex $u\in Q$ such that $\dist(u,w)\leq 2$.
\hfill $\diamond$
\vspace{0.25cm}

Thus, $Q$ is an independent set such that for all $w \in V(D)$ there exists a vertex $u\in Q$ such that $\dist(u,w)\leq 2$.
This completes the proof.
\end{proof}

With Lemma~\ref{Lemma: Alg works} established we are prepared to prove Theorem~\ref{Thm: Main CounterStructure}.

\begin{proof}[Proof of Theorem~\ref{Thm: Main CounterStructure}]
Let $t\geq 1$ be a constant.
and let $\mathcal{H}$ be a hereditary class of digraphs such that for all sourceless digraphs $D \in \mathcal{H}$
there exists a vertex $v \in V(D)$ such that $v \notin N^+(S(v))$ and
\[
t|N^+(v) \cup N^+(S(v))| \geq |S(v)|+1,
\]
or  $v \in N^+(S(v))$ and
\[
t|N^+(v) \cup N^+(S(v))| \geq |S(v)|.
\]
Suppose 
$n$ is the least integer such that there exists a sourceless digraph
$D \in \mathcal{H}$ 
whose smallest quasi-kernels have order greater than $\frac{tn}{t+1}$.
Let $D \in \mathcal{H}$ be such a digraph.

By assumption there exists a vertex
$v \in V(D)$ such that $v \notin N^+(S(v))$ and
\[
t|N^+(v) \cup N^+(S(v))| \geq |S(v)|+1,
\]
or  $v \in N^+(S(v))$ and
\[
t|N^+(v) \cup N^+(S(v))| \geq |S(v)|.
\]
Let $v$ be such a vertex.
Let $Q$ be the output of Algorithm~\ref{alg:main} for inputs $D$ and $v$.
Let $K,R$ and $H$ be defined as Algorithm~\ref{alg:main},
so that $K$ and $H$ are their final iterations during the running of Algorithm~\ref{alg:main}.

Since $\mathcal{H}$ is a hereditary class $H \in \mathcal{H}$.
Moreover, it is easy to verify that $H$ has strictly fewer vertices than $D$, so $H$ has less than $n$ vertices.
Thus, $H$ is a sourceless digraph in $\mathcal{H}$ with less than $n$ vertices.
By the minimality of $n$ this implies $|R| \leq \frac{t|V(H)|}{t+1}$.

By Lemma~\ref{Lemma: Alg works}, $Q$ is a quasi-kernel in $D$, so by our choice of $D$, $|Q| > \frac{tn}{t+1}$.
Observe that $|Q| = |K| + |R|$.
We will show that $|K| \leq \frac{t(n-|V(H)|)}{t+1}$, which implies that 
\[
\frac{tn}{t+1} < |Q| \leq \frac{t(n-|V(H)|)}{t+1} + \frac{t|V(H)|}{t+1} = \frac{tn}{t+1},
\]
a contradiction.

Notice that when first defining $K$ and $H$, 
$|N^+[v] \cup N^+[S(v)]|$ are removed from $D$ to form $H$,
and $|S(v)|$ vertices are added to $K$.
Next, observe that each time the while loop is run, 
exactly one vertex is added to $K$ 
while at least $2$ vertices are removed from $H$.
Finally, observe that if $v \in N^+(S(v))$, then $v$ is not added to $K$ at the end of the algorithm.

Suppose the while loop is run exactly $i$ times in Algorithm~\ref{alg:main}.
Since $v \notin N^+(S(v))$ implies
\[
t|N^+(v) \cup N^+(S(v))| \geq |S(v)|+1,
\]
or  $v \in N^+(S(v))$ implies
\[
t|N^+(v) \cup N^+(S(v))| \geq |S(v)|,
\]
it is easy to verify that
\begin{align*}
    |K| \leq \frac{t|N^+[v] \cup N^+[S(v)]|}{t+1} + i
\end{align*}
while $|V(H)| \leq n - |N^+[v] \cup N^+[S(v)]|-2i$.
Given $t\geq 1$ it follows that
\begin{align*}
    |K| & \leq \frac{t|N^+[v] \cup N^+[S(v)]|}{t+1} + i \\
    & \leq \frac{t}{t+1}\Big(|N^+[v] \cup N^+[S(v)]| + \frac{t+1}{t}i\Big) \\
    & \leq \frac{t}{t+1}(n - V(H))
\end{align*}
This completes the proof that small quasi-kernels exist.
It remains to be shown that such quasi-kernels can be computed in $O(n^4)$ time.

    To begin we store $D$ as a dictionary whose keys are vertices, which point to a vertex's out-neighbourhood and in-neighbourhood, 
    stored as a pair of sets.
    This allows us to look up a fixed vertex $v$'s out-neighbourhood $N^+(v)$ and in-neighbourhood $N^-(v)$ in constant time,
    assuming the vertex $v$ is already known.

    The algorithm works in the obvious greedy manner.
    We identify a vertex $v$ where 
    \[
    t|N^+(v) \cup N^+(S(v))| \geq |S(v)|+1
    \]
    then we feed this vertex $v$ into the input of Algorithm~\ref{alg:main}, 
    recursively applying the same process until the result digraph is empty, then reconstructing
    a quasi-kernel in $D$ by combining the outputs at each step.
    
    Since the number of vertices is decreasing with each application of the first part of 
    Algorithm~\ref{alg:main}, there are at most $O(n)$ greedy steps.
    Each reconstruction step only involves checking the final condition of Algorithm~\ref{alg:main},
    which can be done in $O(n)$ time.
    So reconstructing the final quasi-kernel from the intermediate outputs requires $O(n^2)$ time,
    moreover this is done after all of the intermediate outputs are calculated, so these steps act additively to the time complexity.

    Once $v$ and $S(v)$ is identified, 
    applying the first part of Algorithm~\ref{alg:main} is $O(n)$ time.
    The real time bottleneck then is searching for a vertex $v$ where 
    \[
    t|N^+(v) \cup N^+(S(v))| \geq |S(v)|+1.
    \]
    Doing this naively takes $O(n^3)$ time. 
    
    We explore exactly why now.
    Based on how we have stored $D$ neighbourhood look ups are constant time, but reading a in-neighbourhood or out-neighbourhood 
    might be $O(n)$ time.
    So given a $v$ identifying $S(v)$ is $O(n)$ time, since deleting $N^+[v]$ is linear, and looking for sinks in this data structure is constant time.
    However, checking finding $N^+(S(v))\setminus N^+(v)$ may take $O(|S(v)|n) = O(n^2)$ steps.
    So checking the inequality is $O(n^2)$ time.
    Thus, if we search for a vertex $v$ by checking the inequality for each vertex in the digraph, 
    this will take $O(n^3)$ steps.

    Since we may run linearly many intermediate steps, that is greedy searches,
    the full algorithm may take as much as $O(n^4)$ steps.
    this completes the proof.
\end{proof}

As stated earlier, Theorem~\ref{Thm: Main CounterStructure} provides forbidden configurations in a smallest counterexample to Conjecture~\ref{Conj: small quasi}. We state these explicitly now. By exploiting these forbidden configurations
we obtain bounds on the size of quasi-kernels in Sections~3 and 4.

\begin{corollary}\label{Coro: Smallest}
    Let $\mathcal{H}$ be a hereditary class of digraphs and $t\geq 1$ a constant.
    If $D \in \mathcal{H}$ is a smallest digraph such that for all quasi-kernels $K$ in $D$, $|K| > \frac{tn}{t+1}$,
    then for all vertices $v$ in $D$,
    $
    t|N^+(v) \cup N^+(S(v))| < |S(v)|+1.
    $
\end{corollary}

\begin{proof}
    Suppose $D$ is a smallest digraph in $\mathcal{H}$ with no 
    quasi-kernel of order at most $\frac{tn}{t+1}$.
    Suppose that there exists a vertex $v$ with 
    $
    t|N^+(v) \cup N^+(S(v))| \geq |S(v)|+1.
    $
    in $D$.
    Then apply Algorithm~\ref{alg:main} with inputs $D$ and $v$,
    letting $H$, $K$, $R$, and $Q$ be defined as in Algorithm~\ref{alg:main} for this input.
    Since $D$ is a smallest digraph in $\mathcal{H}$ with this property, 
    $|R|\leq \frac{t|V(H)|}{t+1}$, and by our choice of $v$, $|K|\leq \frac{t(n - |V(H)|)}{t+1}$.
    Therefore, $Q = K\cup R$ is a quasi-kernel of size at most $\frac{tn}{t+1}$.
\end{proof}

The final part of this section is devoted to demonstrating the limitations of Algorithm~\ref{alg:main}.
In particular, we note that Corollary~\ref{Coro: Smallest} is insufficient to prove Conjecture~\ref{Conj: small quasi}.
Formally, we prove that following proposition.
See Figure~\ref{fig:payley sink} for an example of the $c = \frac{3}{2}$ case.

\begin{figure}[h!]
    \centering
    \includegraphics[scale = 0.9]{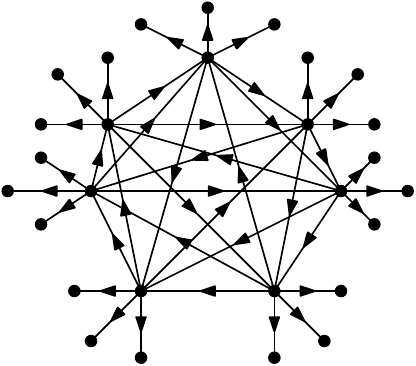}
    \caption{The $7$ vertex Paley Tournament with $3$ degree $1$ sinks appended to each vertex.}
    \label{fig:payley sink}
\end{figure}

\begin{proposition}\label{Prop: Alg not perfect}
    For every constant $c>0$ there exists a sourceless digraph $D$ such that for all
    $v \in V(D)$ with $\deg^+(v)\geq 1$, $c|N^+(v) \cup N^+(S(v))| < |S(v)|+1$.
\end{proposition}

\begin{proof}
    Let $c > 0$ be a fixed but arbitrary constant.
    Let $n$ and $k$ be any integers such that $c(n+k) < nk+1$.
    Trivially, such a pair $n$ and $k$ exists.
    Let $T$ be a tournament with $2n+1$ vertices such that every vertex $v \in V(T)$ has $\deg^+(v) = \deg^-(v) = n$
    and there is no pair of vertices $u$ and $v$ where $N^+(v) = N^-(u)$.
    Given $T$ we define $D$ by adding $k$ vertices $v_1,\dots, v_k$ to $T$ for each vertex $v \in V(T)$
    such that each $v_i$ is a sink with $N^-(v_i) = \{v\}$.

    Since $T$ as sourceless, $D$ is sourceless.
    Additionally, each non-sink in $D$ is a vertex of $T$.
    Hence, for all $v \in V(D)$ with $\deg^+(v)\geq 1$, we note that $\deg^+(v) = n+k$.
    Next, observe that by our assumption that there is no pair of vertices $u$ and $v$ where $N^+(v) = N^-(u)$ in $T$,
    for each $v \in V(D)$ with $\deg^+(v)\geq 1$, $|S(v)| = nk$ and every vertex in $S(v)$ is a sink.
    Thus, for all $v \in V(D)$ with $\deg^+(v)\geq 1$, $N^+(S(v)) = \emptyset$.
    
    It follows that for all
    $v \in V(D)$ with $\deg^+(v)\geq 1$,
    \begin{align*}
        c|N^+(v) \cup N^+(S(v))|  & = c(n+k) \\
        & < nk + 1 \\
        & = |S(v)|+1
    \end{align*}
    by our choice of $n$ and $k$.
    This completes the proof.
\end{proof}

\section{Digraphs with Small Out-Degree}

In this section we will study digraphs with maximum out-degree at most $3$.
Recall that it was shown in \cite{ai2023results} that all $\Vec{K}_{1,3}$-free digraphs,
and therefore all maximum out-degree $2$ digraphs, have a small quasi-kernel.
In this section we consider maximum out-degree $3$ digraphs.
We prove a weaker bound for this class, but
our proof yields a polynomial algorithm for finding such a quasi-kernel.

In a series of lemmas that build on each other, we
will prove a smallest counterexample to Theorem~\ref{Thm: Out-degree 3} must be $3$-out-regular.
That is, all vertices must be out-degree $3$.

\begin{lemma}\label{Lemma: out-3, min-0}
   If $D$ is a sourceless oriented graph with maximum out-degree $\Delta^+(D)\leq 3$,
   and no vertex $v$ such that 
    \[
    \frac{4}{3}|N^+(v) \cup N^+(S(v))| \geq |S(v)|+1,
    \]
    then $D$ has minimum out-degree $\delta^+(D)\geq 1$.
\end{lemma}

\begin{proof}
    Let $D$ be an sourceless oriented graph with maximum out-degree $\Delta^+(D)\leq 3$,
   and no vertex $v$ such that 
    \[
    \frac{4}{3}|N^+(v) \cup N^+(S(v))| \geq |S(v)|+1.
    \]
    For contradiction suppose that there is a sink $x$ in $D$.
    Since $D$ is sourceless there exists a vertex $u \in N^-(x)$.
    Let $x,u$ be such vertices.

    If $\deg^+(u) = 1$, then $S(u) = \emptyset$.
    In this case $v = u$ is a vertex we have forbidden, a contradiction.
    So $\deg^+(u) \geq 2$.

    If $\deg^+(u) = 2$, then let $N^+(u) = \{x,w\}$ without loss of generality.
    Then, $S(u) \subseteq N^+(w)$ implying that $|S(u)| \leq 3$.
    So, $|S(u)| + 1 \leq 4$.
    Notice $|S(u)|\geq 2$, otherwise we can set $v = u$.
    Hence, if $|N^+(u) \cup N^+(S(u))| \geq 3$, then 
    \[
    |S(u)| + 1 \leq 4 \leq \frac{4}{3} |N^+(u) \cup N^+(S(u))|
    \]
    a contradiction, since we can set $v = u$. 
    So we suppose $|N^+(u) \cup N^+(S(u))| < 3$ implying $N^+(S(u)) \subseteq N^+(u)$.
    Since $D$ is oriented and $ux,uw \in E(D)$,
    we note that $x \notin S(w)$.
    Given, $S(u) \subseteq N^+(w)$, and $|S(u)|\geq 2$, while $D$ has maximum out-degree $3$,
    this implies $S(w) \leq 3$.
    Moreover, $S(w)> 0$ only if $w$ has out-degree $3$.
    If $S(w) = \emptyset$, then set $v = w$.
    Otherwise, $\deg^+(w) = 3$, and 
    \[
    |S(w)|+1 \leq 4 \leq \frac{4}{3}|N^+(w)|
    \]
    so we can set $v = w$.
    Since both cases lead to a contradiction
    we conclude that $\deg^+(u) = 3$.

    Without loss of generality let $N^+(u) = \{x,w,z\}$,
    and without loss of generality suppose that $|N^+(w)\cap S(u)|\geq |N^+(z)\cap S(u)|$.
    So $S(u) \subseteq N^+(w) \cup N^+(z)$, meaning $|S(u)|\leq 6$.
    Since $\deg^+(u) = 3$, $|S(u)|\geq 4$,
    moreover, if $N^+(S(u)) \not\subseteq N^+(u)$, then $|S(u)|\geq 5$.
    Thus, $|N^+(w)\cap S(u)| \geq 2$ and $|N^+(w)\cap S(u)| = 3$ if $N^+(S(u)) \not\subseteq N^+(u)$.

    Consider the case where $N^+(w)\subseteq S(u)$.
    As before, $x \notin S(w)$, and by the same argument $z \notin S(w)$,
    so $S(w) \cap N^+(u) = \emptyset$.
    Since, $N^+(w)\subseteq S(u)$, we note that $S(w)\subseteq N^+(S(u))$.
    If $\deg^+(w) = 2$, then $|S(u)| = 4$ implying $N^+(S(u)) \subseteq N^+(u)$.
    But this forces $S(w) = \emptyset$ a contradiction since we can take $v = w$,
    so $\deg^+(w) = 3$.
    Since $\deg^+(w) = 3$ we note that $|S(w)|\geq 4$,
    otherwise we can take $v = w$,
    implying that
    $4 \leq |S(w)| \leq |N^+(S(u))\setminus N^+(u)|$ given $S(w) \cap N^+(u) = \emptyset$.
    But this implies 
    \[
    |S(u)| + 1 \leq 7 < \frac{4}{3}(7) \leq \frac{4}{3}|N^+(u) \cup N^+(S(u))|.
    \]
    Since this would let us take $v = u$, we conclude that $N^+(w)\not\subseteq S(u)$.

    Since $N^+(w)\not\subseteq S(u)$ and $|N^+(w) \cap S(u)|\geq 2$,
    we observe that $|N^+(w)\setminus S(u)| = 1$, and $|N^+(w) \cap S(u)|= 2$, and $|S(u)| = 4$.
    As before, $S(w) \cap N^+(u) = \emptyset$.
    This implies $|S(w)| - 3 \leq |N^+(S(u))|$, given the unique vertex in $N^+(w)\setminus S(u)$
    contributes at most $3$ vertices to $S(w)$.
    Since $|N^+(w)\setminus S(u)| = 1$, and $|N^+(w) \cap S(u)|= 2$ we note that 
    $\deg^+(w) = 3$, and as a result $|S(w)|\geq 4$ otherwise we can choose $v = w$.
    Thus, $1\leq |S(w)| - 3 \leq |N^+(S(u))|$ implying that
    \[
    5 = |S(u)| + 1 < \frac{4}{3}(4) \leq \frac{4}{3}|N^+(u) \cup N^+(S(u))|.
    \]
    Since this would let us take $v = u$, we have completed the proof.   
\end{proof}

\begin{lemma}\label{Lemma: out-3, min-1}
   If $D$ is a sourceless oriented graph with maximum out-degree $\Delta^+(D)\leq 3$,
   and no vertex $v$ such that
    \[
    \frac{4}{3}|N^+(v) \cup N^+(S(v))| \geq |S(v)|+1,
    \]
    then $D$ has minimum out-degree $\delta^+(D)\geq 2$.
\end{lemma}

\begin{proof}
    Let $D$ be an sourceless oriented graph with maximum out-degree $\Delta^+(D)\leq 3$,
   and no vertex $v$ such that 
    \[
    \frac{4}{3}|N^+(v) \cup N^+(S(v))| \geq |S(v)|+1.
    \]
    By Lemma~\ref{Lemma: out-3, min-0} $D$ has minimum out-degree at least $1$.
    For contradiction suppose that there is a vertex $x$ in $D$ with out-degree $1$.
    Since $D$ is sourceless there exists a vertex $u \in N^-(x)$.
    Let $x,u$ be such vertices, and let $\{y\} = N^+(x)$.

    If $\deg^+(u) = 1$, then $S(u) \subseteq \{y\}$.
    Since $D$ has minimum out-degree $1$, $y$ has an out-neighbour $z$.
    Notice $z \neq x$ since $D$ is oriented.
    Regardless of whether $z\neq u$ or $z = u$, it must be the case that $|N^+(v) \cup N^+(S(v))|\geq 2 \geq |S(u)|+1$.
    So we can take $v = u$, a contradiction.

    If $\deg^+(u) = 2$, 
    then let $N^+(u) = \{x,w\}$ without loss of generality.
    Observe that $|S(u)|\leq 4$.
    If $S(u) \subseteq N^+(x)$ or $S(u) \subseteq N^+(w)$, then we reach a contradiction by the same argument used in the proof of 
    Lemma~\ref{Lemma: out-3, min-0}.
    Suppose that $S(u) \not\subseteq N^+(x)$ and $S(u) \not\subseteq N^+(w)$.
    Then, $y \in S(u)$.
    
    Consider the case where $|N^+(S(u))\setminus N^+(u)]| \geq 2$.
    Then, $|S(u)|\geq 5$, otherwise we can let $v = u$.
    But this is a contradiction, so we suppose $|N^+(S(u))\setminus N^+(u)| \leq 1$.

    Since $S(u)$ is a set of sources in $D-N^+[u]$, $S(u)$ is an independent set.
    We conclude that $y \in S(u)$ implies
    \[
    N^+(y) \subseteq \{u,w\} \cup \Big(N^+(S(u))\setminus N^+(u)\Big).
    \]
    Since $x$ has $\deg^+(x) = 1$, $S(x)\subseteq N^+(y)$.
    Notice that since $D$ is oriented,  $uw, ux \in E(D)$ implies $w \notin S(x)$.
    Thus, $|N^+(S(u))\setminus N^+(u)| \leq 1$ implies $|S(x)|\leq 2$, with equality only if $u \in S(x)$.  
    
    If $u \in S(x)$, then $\{x,w\} \subseteq N^+(S(x))$, implying 
    $\frac{4}{3}|N^+(x)\cup N^+(S(x))| \geq 4 > 2\geq |S(x)|$.
    In this case take $v =x$.
    Otherwise, $u \not\in S(x)$ implying that $|S(x)|\leq 1$.
    If $S(x) = \emptyset$, then take $v = x$.
    If $|S(x)| = 1$, then let $q \in S(x)$, since $D$ has minimum out-degree $1$, $\deg^+(q)\geq 1$.
    Since $D$ is oriented and $yq \in E(D)$, $N^+(S(x))\setminus N^+(x) = N^+(q)\neq \emptyset$.
    It follows that 
    \[
    |S(x)|+1 = 2 < \frac{4}{3}(2) \leq \frac{4}{3}|N^+(x)\cup N^+(S(x))|.
    \]
    letting us take $v = x$.
    Since this is a contradiction
    we conclude that $\deg^+(u) = 3$.

    Without loss of generality let $N^+(u) = \{x,w,z\}$,
    and without loss of generality suppose that $|N^+(w)\cap S(u)|\geq |N^+(z)\cap S(u)|$.
    Then $|S(u)|\leq 7$.
    Furthermore,
    since $\deg^+(u) = 3$, $|S(u)|\geq 4$,
    moreover, if $|N^+(S(u))\setminus N^+(u)|\geq 2$, then $|S(u)|\geq 6$.
    Since $|S(u)|\geq 4$ we note that $|N^+(w)\cap S(u)| \geq 2$.
    Additionally, observe that $|N^+(w)\cap S(u)| = 3$ if $|N^+(S(u))\setminus N^+(u)|\geq 2$.

    Consider the case where $N^+(w)\subseteq S(u)$.
    As earlier, $x,z \notin S(w)$, so $S(w) \cap N^+(u) = \emptyset$.
    Since, $N^+(w)\subseteq S(u)$, we note that $S(w)\subseteq N^+(S(u))$.
    If $\deg^+(w) = 2$, then 
    \[
    2 \leq |S(w)| \leq |N^+(S(u))\setminus N^+(u)|
    \]
    otherwise we can choose $v = w$.
    But this is a contradiction, because we have shown $|N^+(S(u))\setminus N^+(u)|\geq 2$ forces 
    $\deg^+(w) = |N^+(w)\cap S(u)| = 3> 2 = \deg^+(w)$.
    So we suppose that $\deg^+(w) = 3$.
    Since $\deg^+(w) = 3$ we note that $|S(w)|\geq 3$,
    otherwise we can take $v = w$,
    implying that
    $3 \leq |S(w)| \leq |N^+(S(u))\setminus N^+(u)|$ given $S(w) \cap N^+(u) = \emptyset$.
    But this implies 
    \[
    |S(u)| + 1 \leq 8 = \frac{4}{3}(6) \leq \frac{4}{3}|N^+(u) \cup N^+(S(u))|.
    \]
    Since this would let us take $v = u$, we conclude that $N^+(w)\not\subseteq S(u)$.

    Since $N^+(w)\not\subseteq S(u)$ and $|N^+(w) \cap S(u)|\geq 2$,
    we observe that $|N^+(w)\setminus S(u)| = 1$, and $|N^+(w) \cap S(u)|= 2$.
    Since $|N^+(w)\cap S(u)|\geq |N^+(z)\cap S(u)|$ this forces $|S(u)|\leq 5$
    with equality only if $y \in S(u)$.
    As before, $S(w) \cap N^+(u) = \emptyset$.
    This implies $|S(w)| - 3 \leq |N^+(S(u))|$, given the unique vertex in $N^+(w)\setminus S(u)$
    contributes at most $3$ vertices to $S(w)$.
    Since $|N^+(w)\setminus S(u)| = 1$, and $|N^+(w) \cap S(u)|= 2$ we note that 
    $\deg^+(w) = 3$, and as a result $|S(w)|\geq 4$ otherwise we can choose $v = w$.
    If $1\leq |S(w)| - 3 \leq |N^+(S(u))|$ implying that if $|S(u)|\leq 4$, then 
    \[
    5 = |S(u)| + 1 < \frac{4}{3}(4) \leq \frac{4}{3}|N^+(u) \cup N^+(S(u))|.
    \]
    Since this would let us take $v = u$, we have conclude that $|S(u)|  = 5$.
    So $y \in S(u)$.

    Since $y \in S(u)$ and $\{y\} = N^+(x)$, we note that $S(x) \subseteq N^+(y) \subseteq N^+(S(u))$.
    As we have seen at several points, in varied contexts, $N^+(u)\cap S(x) = \emptyset$.
    Since $D$ has minimum out-degree $1$ and $D$ is oriented, we observed that
    $S(x) \subseteq N^+(y)$ and $\{y\} = N^+(x)$ implies $N^+(S(x))\setminus N^+(x) \neq \emptyset$.
    Thus, $\frac{4}{3}|N^+(x)\cup N^+(S(x))|> 2$ which forces $|S(x)|\geq 2$ otherwise we may take $v =x$.
    But $|S(x)|\geq 2$ implies that $|N^+(S(u))\setminus N^+(u)|\geq 2$,
    which in turn implies $N^+(w)\subseteq S(u)$.
    This contradicts our earlier conclusion that $N^+(w)\not\subseteq S(u)$,
    thereby completing the proof.
\end{proof}

\begin{lemma}\label{Lemma: out-3, min-2}
   If $D$ is a sourceless oriented graph with maximum out-degree $\Delta^+(D)\leq 3$,
   and no vertex $v$ such that
    \[
    \frac{4}{3}|N^+(v) \cup N^+(S(v))| \geq |S(v)|+1,
    \]
    then $D$ has minimum out-degree $\delta^+(D)\geq 3$.
\end{lemma}

\begin{proof}
    Let $D$ be an sourceless oriented graph with maximum out-degree $\Delta^+(D)\leq 3$,
   and no vertex $v$ such that 
    \[
    \frac{4}{3}|N^+(v) \cup N^+(S(v))| \geq |S(v)|+1.
    \]
    By Lemma~\ref{Lemma: out-3, min-1} $D$ has minimum out-degree at least $2$.
    For contradiction suppose that there is a vertex $x$ in $D$ with out-degree $2$.
    Since $D$ is sourceless there exists a vertex $u \in N^-(x)$.
    Let $x,u$ be such vertices, and let $\{y,t\} = N^+(x)$.

    If $\deg^+(u) = 2$, 
    then let $N^+(u) = \{x,w\}$ without loss of generality.
    Observe that $|S(u)|\leq 5$.
    If $S(u) \subseteq N^+(x)$ or $S(u) \subseteq N^+(w)$, then we reach a contradiction by the same argument used in the proof of 
    Lemma~\ref{Lemma: out-3, min-0}.
    Suppose that $S(u) \not\subseteq N^+(x)$ and $S(u) \not\subseteq N^+(w)$.
    Without loss of generality suppose $y \in S(u)$.

    Consider the case where $|N^+(S(u))\setminus N^+(u)]| \geq 3$.
    Then, $|S(u)|\geq 6$, otherwise we can let $v = u$.
    But this is a contradiction, so we suppose $|N^+(S(u))\setminus N^+(u)| \leq 2$.

    Similarly if $|N^+(S(u))\setminus N^+(u)]| = 2$, then $|S(u)| = 5$ so $t \in S(u)$.
    It follows that $S(x)\subseteq N^+(S(u))$.
    As in earlier lemmas $w \notin S(x)$, hence, $|S(x)|\leq 2$.
    Given $\deg^+(x) = 2$ we note that $|S(x)| \geq 2$ otherwise we can take $v =x$, so $|S(x)| = 2$.
    Again since $\deg^+(x) = 2$, and the fact that $D$ is oriented with minimum out-degree $2$, while $S(x)$ is an independent set,
    we observe that $N^+(S(x))\setminus N^+(x) \neq \emptyset$.
    Thus, 
    \[
    |S(x)|+1 = 3 < \frac{4}{3}(3)\leq \frac{4}{3}|N^+(x)\cup N^+(S(x))|
    \]
    letting us take $v = x$, a contradiction.
    We conclude $|N^+(S(u))\setminus N^+(u)]| \leq 1$.
    By the same argument as $x$, $u$ having out-degree $2$ implies $N^+(S(u))\setminus N^+(u) \neq \emptyset$,
    thus we conclude $|N^+(S(u))\setminus N^+(u)]| = 1$.

    Since $|N^+(S(u))\setminus N^+(u)]| = 1$ we note that $|S(u)|\geq 4$, otherwise we can take $v = u$.
    Hence, $N^+(x) \subseteq S(u)$ or $N^+(w) \subseteq S(u)$.
    Let $\Bar{x}$ be this vertex. Then as should be familiar $S(\Bar{x}) \cap N^+(u) = \emptyset$.
    Since $N^+(\Bar{x}) \subseteq S(u)$, we note $S(\Bar{x}) \subseteq N^+(S(u))$,
    so $|S(\Bar{x})|\leq 1$.
    But $\deg^+(\Bar{x})\geq 2$ since $D$ has minimum out-degree $2$.
    Thus, we can take $v = \Bar{x}$ a contradiction.
    We conclude that $\deg^+(u) = 3$.

    Without loss of generality let $N^+(u) = \{x,w,z\}$,
    and without loss of generality suppose that $|N^+(w)\cap S(u)|\geq |N^+(z)\cap S(u)|$.
    Then $|S(u)|\leq 8$.
    If $|N^+(S(u))\setminus N^+(u)|\geq 4$, then
    \[
    |S(u)|+1 \leq 9 < \frac{4}{3}(7) \leq \frac{4}{3}|N^+(u) \cup N^+(S(u))|
    \]
    letting us take $v = u$, so $|N^+(S(u))\setminus N^+(u)|\leq 3$.
    Consider the case where $N^+(x)\subseteq S(u)$.
    Then,  $S(x) \subseteq N^+(S(u))\setminus N^+(u)$ implying $|S(x)|\leq 3$.
    Since $x$ has degree $2$, and $D$ has minimum degree $2$, $N^+(S(x))\setminus N^+(x) \neq \emptyset$.
    Hence, 
    $|S(x)|+1 \leq 4 \leq \frac{4}{3}(3) = \frac{4}{3}|N^+(x) \cup N^+(S(x))|$
    allowing us to take $v = x$ a contradiction.
    We conclude $N^+(x)\subseteq S(u)$, which implies $|S(u)|\leq 7$.
    Since $|S(u)|\leq 7$ we note that $|N^+(S(u))\setminus N^+(u)|= 3$ implies 
    \[
    |S(u)|+1 \leq 8 = \frac{4}{3}(6) \leq \frac{4}{3}|N^+(u) \cup N^+(S(u))|
    \]
    letting us take $v = u$, so $|N^+(S(u))\setminus N^+(u)|\leq 2$.
    Finally, observe that since $N^+(x)\not\subset S(u)$, 
    if $|N^+(S(u))\setminus N^+(u)|= 2$, implies  $|N^+(w)\cap S(u)| = 3$.

    Consider the case where $N^+(w)\subseteq S(u)$.
    As earlier, $x,z \notin S(w)$, so $S(w) \cap N^+(u) = \emptyset$.
    Since, $N^+(w)\subseteq S(u)$, we note that $S(w)\subseteq N^+(S(u))$.
    If $\deg^+(w) = 2$, then 
    \[
    2 \leq |S(w)| \leq |N^+(S(u))\setminus N^+(u)|
    \]
    otherwise we can choose $v = w$.
    But this is a contradiction, because we have shown $|N^+(S(u))\setminus N^+(u)|\geq 2$ forces 
    $\deg^+(w) = |N^+(w)\cap S(u)| = 3> 2 = \deg^+(w)$.
    So we suppose that $\deg^+(w) = 3$.
    As before, $S(w)\cap N^+(u) = \emptyset$ and $S(w) \subseteq N^+(S(u))$.
    Since $\deg^+(w) = 3$ we note that $|S(w)|\geq 4$,
    otherwise we can take $v = w$,
    implying that
    $4 \leq |S(w)| \leq |N^+(S(u))\setminus N^+(u)|$ given $S(w) \cap N^+(u) = \emptyset$.
    But this implies 
    \[
    |S(u)| + 1 \leq 8 < \frac{4}{3}(7) \leq \frac{4}{3}|N^+(u) \cup N^+(S(u))|.
    \]
    Since this would let us take $v = u$, we conclude that $N^+(w)\not\subseteq S(u)$.

    Since, $N^+(x)\not\subseteq S(u)$, and 
    $|N^+(w)\cap S(u)|\geq |N^+(z)\cap S(u)|$ we note that if $|N^+(w)\cap S(u)|\leq 1$, then $|S(u)|\leq 3$.
    This would allow us to take $v = u$ since $\deg^+(u) = 3$,
    so we suppose $|N^+(w)\cap S(u)|\geq 2$.
    From here $N^+(w)\not\subseteq S(u)$ implies that $\deg^+(w) = 3$ and $|N^+(w)\cap S(u)|= 2$.
    Now, $|N^+(w)\cap S(u)|= 2$, $N^+(x)\not\subseteq S(u)$, and 
    $|N^+(w)\cap S(u)|\geq |N^+(z)\cap S(u)|$ implies $|S(u)|\leq 5$.

    As before, $S(w) \cap N^+(u) = \emptyset$.
    This implies $|S(w)| - 3 \leq |N^+(S(u))|$, given the unique vertex in $N^+(w)\setminus S(u)$
    contributes at most $3$ vertices to $S(w)$.
    Since $\deg^+(w) = 3$ we note that $|S(w)|\geq 4$,
    otherwise we can take $v = w$,
    implying that $1 \leq |S(w)| - 3 \leq |N^+(S(u))\setminus N^+(u)|$.
    If $|S(u)|\leq 4$, then 
    \[
    5 = |S(u)| + 1 < \frac{4}{3}(4) \leq \frac{4}{3}|N^+(u) \cup N^+(S(u))|.
    \]
    allowing us to take $v = u$.
    Hence, we conclude that $|S(u)| = 5$ which forces $N^+(x)\cap S(u) \neq \emptyset$.
    Since $|S(u)|\leq 5$ we note that $|N^+(S(u))\setminus N^+(u)| = 1$, otherwise we can take $v = u$.
    Without loss of generality suppose that $y \in S(u)$ and $t\notin S(u)$.

    As before, $S(x) \cap N^+(u) = \emptyset$.
    This implies $|S(x)| - 3 \leq |N^+(S(u))\setminus N^+(u)| = 1$, given the unique vertex in $N^+(x)\setminus S(u)$, so $|S(x)|\leq 4$.
    Since $\deg^+(x) = 2$ and $D$ has minimum out-degree $2$, we note that $|N^+(S(x))\setminus N^+(x)|\geq 1$.
    It follows that $|S(x)| \geq 4$, implying $|S(x)|= 4$, which forces
    $|N^+(S(x))\setminus N^+(x)|= 1$ and $|N^+(t) \cap S(x)| = 3$.
    Since $|N^+(t) \cap S(x)| = 3$, $N^+(t) \subseteq S(x)$ and $\deg^+(t) = 3$.
    But $\deg^+(t) = 3$ implies $|S(t)|\geq 4$ while $N^+(t) \subseteq S(x)$ implies
    $S(t) \subseteq N^+(S(x))\setminus N^+(x)$.
    This implies $4 \leq |S(t)| \leq |N^+(S(x))\setminus N^+(x)| =1$ a contradiction.
    This completes the proof.
\end{proof}

We are now prepared to prove Theorem~\ref{Thm: Out-degree 3}.
To assist the reader we provide Figure~\ref{fig:3-reg}.
This figure is also illustrative of ideas used earlier in the section.

\begin{figure}[H]
\centering
\scalebox{0.9875}{
\begin{subfigure}{0.45\textwidth}
    \includegraphics[scale = 1.0]{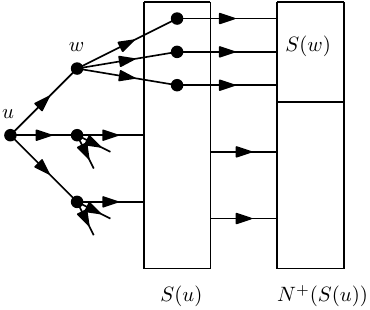}
    \caption{}
    \label{fig:3-ref contained}
\end{subfigure}
\begin{subfigure}{0.45\textwidth}
    \includegraphics[scale = 1.0]{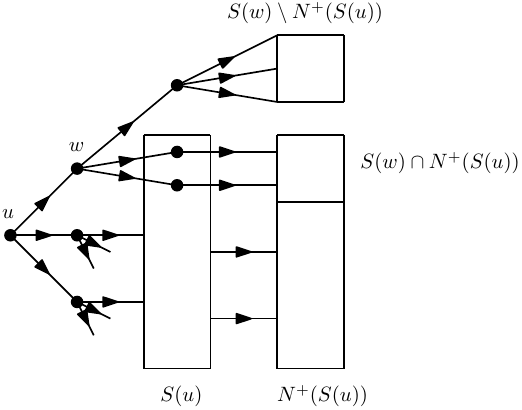}
    \caption{}
    \label{fig:3-ref NOTcontained}
\end{subfigure}
}
    \caption{Two figures which display the key cases covered to prove Theorem~\ref{Thm: Out-degree 3}.}
    \label{fig:3-reg}
\end{figure}

\begin{proof}[Proof of Theorem~\ref{Thm: Out-degree 3}]
    We will begin by proving that all sourceless digraphs with out-degree at most $3$ have a quasi-kernel of order at most
    $\frac{4n}{7}$, then we will prove that such a quasi-kernel can always be found in $O(n^2)$ time.
    Suppose for contradiction that the first claim is false.
    
    Let $D$ be a smallest sourceless digraph with maximum out-degree $\Delta^+ \leq 3$
    and no quasi-kernel of order at most $\frac{4n}{7}$.
    Without loss of generality $D$ is an oriented digraph.
    Then, Corollary~\ref{Coro: Smallest} implies there is no vertex $v$ in $D$ such
    that 
    \[
    \frac{4}{3}|N^+(v) \cup N^+(S(v))| \geq |S(v)|+1.
    \]
    Then by Lemma~\ref{Lemma: out-3, min-2} $D$ has minimum out-degree $\delta^+ \geq 3$,
    implying $D$ is $3$-out-regular.

    Let $u$ be a fixed in $D$ that minimizes $|S(u)|$,
    and let $N^+(u) = \{x,w,z\}$.
    Then, $|S(u)|\leq 9$.
    Without loss of generality suppose $|N^+(w)\cap S(u)|\geq |N^+(z)\cap S(u)|\geq |N^+(x)\cap S(u)|$.
    Observe that $|N^+(S(u))\setminus N^+(u)| \geq 5$ implies
    that
    \[
    |S(u)|+1 \leq 10 < \frac{4}{3}(8) \leq \frac{4}{3}|N^+(u)\cup N^+(S(u))|
    \]
    a contradiction.
    So $|N^+(S(u))\setminus N^+(u)| \leq 4$.

    Consider the case where $N^+(w) \subseteq S(u)$.
    See Figure~\ref{fig:3-ref contained} for an illustration of this case.
    Since $N^+(w) \subseteq S(u)$, we note that $S(w) \subseteq N^+(S(u))\setminus N^+(u)$.
    Given $D$ is $3$-out-regular and oriented $N^+(S(w))\setminus N^+(w) \neq \emptyset$.
    Hence, 
    $\frac{4}{3}|N^+(w)\cup N^+(S(w))|\geq \frac{4}{3}(4)>5$
    which implies $|S(w)|\geq 5$.
    Thus, 
    $N^+(w) \subseteq S(u)$ implies $|N^+(S(u))\setminus N^+(u)| \geq 5$
    a contradiction.
    So $N^+(w) \not\subseteq S(u)$
    implying that $|S(u)|\leq 6$.
    See Figure~\ref{fig:3-ref NOTcontained} for an illustration of this case.

    Since $D$ is $3$-out-regular and oriented $N^+(S(u))\setminus N^+(u) \neq \emptyset$.
    This forces $|S(u)|\geq 5$.
    By our choice of $w$ and the fact that $N^+(w) \not\subseteq S(u)$
    we note that $|N^+(w)\cap S(u)| = 2$
    Hence, recalling that $|S(w)|\geq 5$,
    $
    2 \leq |S(w)|-3 \leq |N^+(S(u))\setminus N^+(u)|.
    $
    Notice that this forces 
    \[
    6 < \frac{4}{3}(5) \leq \frac{4}{3}|N^+(u)\cup N^+(S(u))| < |S(u)|+1
    \]
    so $|S(u)| = 6$.
    By our choice of $u$, $|S(w)|\geq |S(u)| = 6$, 
    so $|S(w)|\geq 6$ in turn implies that 
    $
    3 \leq |S(w)|-3 \leq |N^+(S(u))\setminus N^+(u)|
    $
    forcing
    \[
    7 = |S(u)|+ 1< \frac{4}{3}(6) \leq \frac{4}{3}|N^+(u)\cup N^+(S(u))| < |S(u)|+1
    \]
    a contradiction.
    We conclude that our assumption 
    $D$ is a smallest sourceless digraph with maximum out-degree $\Delta^+ \leq 3$
    and no quasi-kernel of order at most $\frac{4n}{7}$ is false.
    Thus, no such counterexample exists.
\end{proof}

\section{Restricting Short Cycles}

In this section we consider
digraphs with restricted short cycles.
In particular we will prove Theorem~\ref{Thm: C_3,C_4,C_6}.
We require several lemmas and the introduction of some new notation for this proof.
Given a digraph $D$ and a vertex subset $X \subseteq V(G)$, we define $T(X) \subseteq S(X)$ to be the set of isolated vertices in $D-N^+[X]$.
For a vertex $v$, we let denote $T(\{v\})$ by $T(v)$.

\begin{lemma}\label{Lemma: no short-cycle S-struct}
    Let $D$ be a sourceless oriented graph with no $C_4^{\Uparrow}$ as a subgraph.
    For all vertices $v \in V(D)$, if $u \in S(v)$, then $\deg^-(u) = 1$.
\end{lemma}

\begin{proof}
    Let $D$ be a sourceless oriented graph with no $\Vec{C}_3, \Vec{C}^{-}_4,C_4^{\Uparrow},$ or $C_6^{\Uparrow}$ as a subgraph.
    Let $v \in V(D)$ be a fixed but arbitrary vertex and let $u \in S(v)$.
    For contradiction suppose $\deg^-(u)\geq 2$.
    Since $u \in S(v)$, $N^-(u) \subseteq N^+(v)$.
    Hence, $\deg^-(u)\geq 2$ implies there exists vertex $w,z \in N^-(u)\cap N^+(v)$.
    Notice this implies $\{u,v,w,z\}$ contains a copy of $C_4^{\Uparrow}$, a contradiction.
\end{proof}

\begin{lemma}\label{Lemma: no short-cycle S-neighbourhoods}
    Let $D$ be a sourceless oriented graph with no $\Vec{C}_3, \Vec{C}^{-}_4,C_4^{\Uparrow},$ or $C_6^{\Uparrow}$ as a subgraph.
    For all vertices $v \in V(D)$, if $u,w \in S(v)$, then $N^+(u)\cap N^+(w) = \emptyset$ and $N^+(u)\cup N^+(w) \subseteq N^+(S(v))\setminus N^+[v]$.
\end{lemma}

\begin{proof}
    Let $D$ be a sourceless oriented graph with no $\Vec{C}_3, \Vec{C}^{-}_4,C_4^{\Uparrow},$ or $C_6^{\Uparrow}$ as a subgraph.
    Let $v \in V(D)$ be a fixed but arbitrary vertex and let $u,w \in S(v)$.
    For contradiction suppose $N^+(u)\cap N^+(w) \neq \emptyset$.
    Then there exists a vertex $z \in N^+(u)\cap N^+(w)$.
    Moreover, since $u,w \in S(v)$, there exists not-necessary distinct vertices $x \in N^-(u)\cap N^+(v)$
    and $y \in N^-(w)\cap N^+(v)$.
    Since $D$ is oriented $z \notin \{x,y\}$
    
    If $z  = v$, then $vxu$ is a copy $\Vec{C}_3$, a contradiction.
    So $z \neq v$.
    Suppose $z \in N^+(v)$.
    In this case $\{u,v,x,z\}$ contains a copy of $\Vec{C}^{-}_4$, a contradiction.
    So $z \notin N^+(v)$.
    It follows $N^+(u)\cup N^+(w) \subseteq N^+(S(v))\setminus N^+(v)$.

    If $x = y$, then $\{x,u,w,z\}$ contains a copy of $C_4^{\Uparrow}$,
    a contradiction. So suppose $x\neq y$.
    But if $x\neq y$, then $\{v,x,y,u,w,z\}$ contains a copy of $C_6^{\Uparrow}$, since these are all distinct vertices.
    As this is also a contradiction, the proof is complete.
\end{proof}

\begin{lemma}\label{Lemma: T-set}
    Let $D$ be an oriented graph with no $\Vec{C}_3, \Vec{C}^{-}_4,C_4^{\Uparrow},$ or $C_6^{\Uparrow}$ as a subgraph.
    For all vertex $v \in V(D)$,
    \[
    |N^+(S(v))\setminus N^+(v)| \geq |S(v)\setminus T(v)|.
    \]
\end{lemma}

\begin{proof}
    Let $D$ be a sourceless oriented graph with no $\Vec{C}_3, \Vec{C}^{-}_4,C_4^{\Uparrow},$ or $C_6^{\Uparrow}$ as a subgraph.
    Let $v \in V(D)$ be a fixed but arbitrary vertex.
    Since $D$ is sourceless, there are no isolated vertices in $D$.
    Hence, any vertex $u \in T(v)$ must have $N^-(u)\cup N^+(u) \subseteq N^+(v)$.
    It follows that $T(v) \subseteq S(v)$.
    
    By Lemma~\ref{Lemma: no short-cycle S-neighbourhoods}
    if $u,w \in S(v)$ are distinct vertices, then $N^+(u)\cap N^+(w) = \emptyset$ and $N^+(u)\cup N^+(w) \subseteq N^+(S(v))\setminus N^+[v]$.
    Since $N^+(u)\cup N^+(w) \subseteq N^+(S(v))\setminus N^+[v]$,
    if $\deg^+(u)\geq 1$ (or $\deg^+(w)\geq 1$), then $u$ (or $w$) is not an isolated vertex in $D-N^+[v]$.
    Hence, $u \in S(v)$ is in $T(v)$ if and only if $\deg^+(u) = 0$.
    Thus, $|N^+(S(v))\setminus N^+(v)| \geq |S(v)\setminus T(v)|$
    completing the proof.
\end{proof}

With these structural lemmas in hand we are prepared to prove Theorem~\ref{Thm: C_3,C_4,C_6}.

\begin{proof}[Proof of Theorem~\ref{Thm: C_3,C_4,C_6}]
    Let $d\geq 2$ be an integer,
    let $D$ a sourceless digraph with maximum out-degree $\Delta^+(D) \leq d$, and no $\Vec{C}_3, \Vec{C}^{-}_4,C_4^{\Uparrow},$ or $C_6^{\Uparrow}$ as a subgraph.
    Suppose, $D$ is a smallest counterexample.
    Without loss of generality $D$ is oriented.
    Then Corollary~\ref{Coro: Smallest}
    implies that 
    for all vertices $v$ in $D$,
    \[
    \frac{d^2+4}{4d}|N^+(v) \cup N^+(S(v))| < |S(v)|+1.
    \]
    For each vertex $v$, let $k_v  = |\{u \in N^+(v): \deg^-(u) = 1, \deg^+(u) = 0\}|$.

    Suppose $v$ is a vertex such that $k_v = \max_{u\in V(D)} k_u$ and $\deg^+(v)\geq 1$.
    If $\max_{u\in V(D)} k_u > 0$ trivially such a vertex exists.
    If $\max_{u\in V(D)} k_u = 0$, then $k_u = 0$ for all vertices, and the fact that $D$ is sourceless implies
    there exists a vertex in $D$ with positive out-degree, so again $v$ exists.
    If $v$ exists and $k_v = 0$, 
    then Lemma~\ref{Lemma: no short-cycle S-struct}
    implies $T(v) = \emptyset$ and Lemma~\ref{Lemma: T-set} implies
    \begin{align*}
        |S(v)\setminus T(v) |+1 &= |S(v)|+1 \\
        & \leq 1+|N^+(S(v))\setminus N^+(v)| \\
        & \leq |N^+(v) \cup N^+(S(v))| \\
        &< |S(v)|+1
    \end{align*}
    a contradiction. So we conclude that $k_ v > 0$. 
    Observe that if $\deg^+(v) = k_v$, then $S(v) = \emptyset$, a contradiction.
    So $1 \leq k_v \leq \deg^+(v)-1$ forcing $\deg^+(v)\geq 2$.

    Let $A(v) = \{w\in N^+(w): \deg^+(w)\geq 1\}$.
    Then $1\leq |A(v)| = \deg^+(v)-k_v$, and $2\leq \deg^+(v) \leq d$.
    Since $d\geq 2$, $\frac{d^2+4}{4d} \geq 1$,
    so Lemma~\ref{Lemma: T-set} implies
    \[
    \frac{d^2+4}{4d}\Big(\deg^+(v)+|S(v)\setminus T(v)| \Big) \leq \frac{d^2+4}{4d}|N^+(v) \cup N^+(S(v))| < |S(v)|+1
    \]
    from which we conclude that $\frac{d^2+4}{4d}\deg^+(v) < |T(v)|+1$ since $T(v) \subseteq S(v)$.
    Considering the expected size of $|T(v)\cap N^+(w)|$ for $w \in A(v)$,
    noting that for $d\geq 2$ the function $\frac{d^2+4}{4d}$ is monotone increasing,
    we see that
    \begin{align*}
        \mathbb{E}(|T(v)\cap N^+(w)|) &= \frac{|T(v)|}{|A(v)|} \\
        & > \frac{1}{\deg^+(v)-k_v}\Bigg(\frac{d^2+4}{4d}\deg^+(v) - 1\Bigg) \\
        & \geq \frac{1}{\deg^+(v)-k_v}\Bigg(\Big(\frac{\deg^+(v)^2}{4}+1 \Big)\frac{\deg^+(v)}{\deg^+(v)}  - 1\Bigg) \\
        & \geq \frac{\deg^+(v)^2}{4(\deg^+(v)-k_v)}
    \end{align*}
    Now consider the real function $f(x,y) = \frac{y^2}{4(y-x)}$.
    Notice that 
    \[
    4x(x-y) + y^2 = 4x^2-4xy+y^2 = (2x-y)^2\geq 0
    \]
    for any reals $x$ and $y$. Thus, for any reals $x,y$
    we have $y^2 \geq 4x(y-x)$ implying in turn that
    $f(x,y) = \frac{y^2}{4(y-x)} \geq x$ for all reals $x,y$ on the functions domain.
    Since we have already shown $\deg^+(v)\neq k_v$, we conclude that
    \[
     \mathbb{E}(|T(v)\cap N^+(w)|) > \frac{\deg^+(v)^2}{4(\deg^+(v)-k_v)} \geq k_v
    \]
    so there exists a vertex $w \in N^+(v)$ with $|T(v)\cap N^+(w)| > k_v$.
    
    By Lemma~\ref{Lemma: no short-cycle S-struct} and Lemma~\ref{Lemma: no short-cycle S-neighbourhoods}
    each vertex in $T(v)\cap N^+(w)$ has in-degree $1$ and out-degree $0$.
    Thus, there exists a vertex $w \in N^+(v)$ with $k_w > k_v = \max_{u\in V(D)} k_u$.
    Since this is a contradiction, we conclude that $D$ is not a counterexample.
    This completes the proof that $D$ contains a quasi-kernel of order at most $\frac{(d^2+4)n}{(d+2)^2}$.
    It remains to be shown that 
    such a quasi-kernel can be found in $O(n^2)$ time.

    To begin we store $D$ as a dictionary whose keys are vertices, which point to a vertex's out-neighbourhood and in-neighbourhood, 
    stored as a pair of sets.
    This allows us to look up a fixed vertex $v$'s out-neighbourhood $N^+(v)$ and in-neighbourhood $N^-(v)$ in constant time,
    assuming the vertex $v$ is already known.
    
    The algorithm to generate a small quasi-kernel in digraphs with no 
    $\Vec{C}_3, \Vec{C}^{-}_4,C_4^{\Uparrow},$ or $C_6^{\Uparrow}$ as a subgraph
    is similar to the $O(n^4)$ time algorithm
    presented during the proof of Theorem~\ref{Thm: Main CounterStructure}.
    That is, we identify a vertex $v$ where 
    \[
    \frac{d^2+4}{4d}|N^+(v) \cup N^+(S(v))| \geq |S(v)|+1
    \]
    then we feed this vertex $v$ into the input of Algorithm~\ref{alg:main}, 
    recursively applying the same process until the result digraph is empty, then reconstructing
    a quasi-kernel in $D$ by combining the outputs at each step.
    
    The time advantage for digraphs with no $\Vec{C}_3, \Vec{C}^{-}_4,C_4^{\Uparrow},$ or $C_6^{\Uparrow}$ as a subgraph
    is the means by which we search for the next vertex $v$
    with 
    \[
    \frac{d^2+4}{4d}|N^+(v) \cup N^+(S(v))| \geq |S(v)|+1.
    \]
    Rather than checking each vertex of the digraph in order until a vertex $v$
    is identified, we begin by choosing an arbitrary vertex $u$ as the active vertex.
    After this first vertex $u$ is identified, 
    we check if $u$ works as the vertex $v$ we require.
    This check only takes 
    which only takes constant time to check since
    we must only examine vertices $x$ such that $\dist(u,x)\leq 3$,
    and
    the size of $u$'s closed third out-neighbourhood is at most $d^3+d^2+d+1$, where $d$ is constant.
    Moreover, for a fixed $w$ in the second out-neighbourhood of $u$, the time it takes to check if $N^-(w)\subseteq N^+(u)$
    is only depends on $d$, because looking up $N^+(u)$ and $N^-(w)$ is constant time, 
    $|N^+(u)|\leq d$, while we can halt the computation early should $|N^-(w)|>d$.
    So we suppose the active vertex $u$ cannot be taken as the desired vertex $v$.
    If the active vertex $u$ has $k_u = 0$, then 
    Lemma~\ref{Lemma: no short-cycle S-struct} and Lemma~\ref{Lemma: T-set}
    implies we can take $v = u$, so $k_u\geq 1$.
    Since, $k_u \geq 1$ and we cannot take $u$ as our desired vertex $v$,
    the above averaging argument implies $u$ has an out-neighbour $w$ with $k_w>k_u$.
    Take $w$ as the active vertex and check if $w$ can work as $v$, if not then 
    the averaging argument works again to find a vertex $z$ with larger $k_z$. 
    Make $z$ the active vertex and repeat until the active vertex satisfies the requirements of $v$.
    Notice that each step takes constantly many steps, because $d$ is constant, 
    while this loop continues for at most $d$ rounds, since the $k$ parameter is increasing but at most $d$.
    Thus, finding a vertex $v$ as desired only requires constantly many steps.

    From here running Algorithm~\ref{alg:main} until $H$ is defined takes $O(n)$ time,
    and $H$ has less vertices than $D$, so Algorithm~\ref{alg:main} and the search for a new vertices $v$ must be carried out at most $O(n)$ times.
    Finally, reconstructing the quasi-kernel in $D$ from the intermediate outputs only requires $O(n^2)$ steps.
    It follows the algorithm will complete in $O(n^2)$ time.
    This completes the proof.
\end{proof}

\section{Future Work}

We conclude with a discussion of future work.
Given the nature of the paper, we will focus on questions and conjectures regarding how the methods introduced here
might be used in the future.
To begin, we conjecture that the proof of Theorem~\ref{Thm: Out-degree 3}
can be strengthens as follows.
If this conjecture is false, it would be interesting to understand
the structure of the counterexample.

\begin{conjecture}
    If $D$ is a sourceless oriented graph with maximum out-degree at most $3$, then 
    there exists a vertex $v \in V(D)$ such that
    \[
    |N^+(v)\cup N^+(S(v))|\geq |S(v)|+1.
    \]
\end{conjecture}

In the same vein it is natural to ask
how Theorem~\ref{Thm: Out-degree 3} can be generalized to digraphs with out-degree at most $d$.
We know that this kind of argument cannot go all the way to proving the Conjecture~\ref{Conj: small quasi}
due to Proposition~\ref{Prop: Alg not perfect}.
What is the limit to this type of argument as $d$ grows?
In particular, what is the smallest $d$ such that these methods cannot prove Conjecture~\ref{Conj: small quasi}?

\begin{problem}
    Let $d\geq 2$ be fixed but arbitrary.
    What is the least constant $f(d)$ such that
    for all  sourceless oriented graphs $D$ with maximum out-degree at most $d$, there 
    exists a vertex $v \in V(D)$ such that
    \[
    f(d)|N^+(v)\cup N^+(S(v))|\geq |S(v)|+1.
    \]
    What is the least $d$ such that $f(d)>1$?
\end{problem}

Given the work here exposes some local structure a smallest counterexample 
to Conjecture~\ref{Conj: small quasi} must satisfy,
it is natural to ask what other structure does this force upon a smallest counterexample?
What structures does Corollary~\ref{Coro: Smallest} preclude?
Perhaps one might prove Conjecture~\ref{Conj: small quasi}, 
at least for a well structured hereditary classes,
by introducing another approach for handling smallest counterexamples,
then showing the resulting structure is not compatible with Corollary~\ref{Coro: Smallest}.

\begin{problem}
    Suppose $D$ is a sourceless oriented graph such that 
    for all $v \in V(D)$ 
    \[
    |N^+(v)\cup N^+(S(v))|< |S(v)|+1.
    \]
    What structure, if any, must $D$ have?
\end{problem}

The final open problem we propose
it to quantify the worst case behaviour of Algorithm~\ref{alg:main}.
Not in terms of time complexity, but in terms of the size of a quasi-kernel outputted by the algorithm.
We have seen in Proposition~\ref{Prop: Alg not perfect}
that there exists digraphs which we cannot theoretically guarantee
Algorithm~\ref{alg:main} will output a small quasi-kernel for.
However, if one examines the tournament construction we provide to prove 
Proposition~\ref{Prop: Alg not perfect} it is clear that 
Algorithm~\ref{alg:main} still outputs a reasonably small quasi-kernel.
We believe this will not happen in general.

\begin{conjecture}
    For all $\epsilon>0$, there exists a sourceless digraph $D$
    such that for all vertices $v$ in $D$, any quasi-kernel returned by Algorithm~\ref{alg:main}
    for inputs $D$ and $v$ has order at least $(1-\epsilon)n$.
\end{conjecture}

\section*{Acknowledgement}

I would like to Paul Seymour for introducing me to this problem and for encouraging my initial efforts to work on it
during the
Banff International Research Station workshop on New Perspectives in Colouring and Structure in October 2024.  

\bibliographystyle{abbrv}
\bibliography{bib}

\end{document}